\documentclass[12pt]{article}
\usepackage{amsmath,amssymb,amsthm,amsfonts,amsbsy,mathtools,enumerate}
\usepackage{tikz}
\usepackage{booktabs}
\usepackage[margin=2cm]{geometry}

\usepackage[titletoc,title]{appendix}

\theoremstyle{plain}
\newtheorem{theorem}{Theorem}[section]
\newtheorem{corollary}[theorem]{Corollary}
\newtheorem{conj}[theorem]{Conjecture}

\newtheorem{prop}[theorem]{Proposition}

\newcommand{\sr}[1]{\mathrm{SR}_{#1}}
\renewcommand{\leq}{\leqslant}
\renewcommand{\geq}{\geqslant}
\newcommand{\cart}{\ensuremath{\mathbin}{\square}}

\begin{document}

\title{On the flip graphs on perfect matchings of complete graphs and signed reversal graphs}
\author{Sebastian M. Cioab\u{a}\footnote{Department of Mathematical Sciences, University of Delaware, Newark, DE 19716-2553, USA {\tt cioaba@udel.edu}. This research has been partially supported by NSF grants DMS-1600768 and CIF-1815922 and a JSPS Invitational Fellowship for Research in Japan S19016.},\,\, Gordon Royle\footnote{Centre for the Mathematics of Symmetry and Computation, Dept. of Mathematics and Statistics, The University of Western Australia, {\tt gordon.royle@uwa.edu.au}} \, and Zhao Kuang Tan\footnote{Nanyang Technological University, Singapore {\tt zhaokuang.tan@ntu.edu.sg}. Part of this work was done while this author visited University of Delaware in 2019 supported by Nanyang Technological University through the CN Yang Scholars Programme.}}
\date{\today}
\maketitle

\begin{abstract}
In this paper, we study the flip graph on the perfect matchings of a complete graph of even order. We investigate its combinatorial and spectral properties including connections to the signed reversal graph and we improve a previous upper bound on its chromatic number.
\end{abstract}

\section{Introduction}

Our graph theory notation is standard (see \cite{BH} for example). When $G$ is a graph of even order $2n$, let $\mathcal{M}(G)$ be the graph whose vertices are the perfect matchings of $G$, and where two perfect matchings are adjacent if their symmetric difference is a cycle of length $4$. This is called the \emph{flip graph} of the set of matchings, because adjacent perfect matchings in $\mathcal{M}(G)$ are related by a \emph{flip move}, which replaces an independent pair of edges in a perfect matching with one of the two different independent pairs of edges on the same four vertices.

In this paper, we will be interested in the graph $\mathcal{M}(K_{2n})$, where $K_{2n}$ is the complete graph on $2n$ vertices. It is not hard to see that $\mathcal{M}(K_{2n})$ is a connected $n(n-1)$-regular graph on $(2n-1)!!=\prod_{j=1}^{n}(2j-1)$ vertices. We can understand the adjacency relation of this graph in terms of certain integer partitions. If $M$ and $M'$ are two perfect matchings in $K_{2n}$, then the multigraph union $M\cup M'$ is a disjoint union of cycles (where length $2$ cycles arise from edges in both $M$ and $M'$). The lengths $2\lambda_1\geq \cdots \geq 2\lambda_k$ of these cycles form a partition of $2n$. In this case $\lambda_1\geq \cdots \geq \lambda_k$ is a partition of $n$ which is usually written as $(\lambda_1,\dots,\lambda_k)\vdash n$. We call $(\lambda_1,\dots,\lambda_k)$ the \emph{partition type} of the pair $(M,M')$. So the matchings $M$ and $M'$ are adjacent in $\mathcal{M}(K_{2n})$ if the partition type of $(M,M')$ is $(2,1^{n-2})$. 

The spectral and combinatorial properties of the graph $\mathcal{M}(K_{2n})$ have been investigated by various authors. Diaconis and Holmes \cite{DH1} studied the connection between $\mathcal{M}(K_{2n})$ and phylogenetic trees. Using the theory of Gelfand pairs and representation theory of the symmetric group, Diaconis and Holmes \cite{DH2} determined the spectra of these graphs and showed that the mixing time of a random walk on $\mathcal{M}(K_{2n})$ exhibits cut-off phenomenon at $\frac{n\log(n)}{2}$ steps. The graph $\mathcal{M}(K_{2n})$ is part of the perfect matching association scheme whose graphs correspond to the integer partitions $(\lambda_1,\dots,\lambda_k)\vdash n$. This association scheme has interesting properties (see Godsil and Meagher \cite[Chapter 15]{GM} and Srinivasan \cite{Srinivasan}) and has been studied in the context of Erd\"{o}s-Ko-Rado theorems for matchings (see Godsil and Meagher \cite{GMpaper}, Lindzey \cite{Lindzey,Lindzey2,Lindzey3} or Ku and Wong \cite{KuWong}). Jennings \cite{Jennings} studied geodesics in the graph $\mathcal{M}(K_{2n})$ and showed that the distance between two perfect matchings $M_1$ and $M_2$ equals $n-c(M_1,M_2)$, where $c(M_1,M_2)$ equals the number of components in the graph $M_1\cup M_2$. Using this result, Jennings proved that $\mathcal{M}(K_{2n})$ has diameter $n-1$ and obtained a formula for the number of geodesics between any two perfect matchings. In particular, Jennings showed that the number of geodesics between any two vertices at distance $n-1$ equals $n^{n-2}$. Hernando, Hurtado and Noy \cite{HHN} studied a geometric version of the graph $\mathcal{M}(K_{2n})$ in which the vertex set consists of the perfect matchings on the nodes of a convex $2n$-gon whose edges are straight lines and do not cross. For $n\geq 2$, they observed that the number of vertices of this graph equals the Catalan number $C_n=\frac{{2n\choose n}}{n+1}$ and proved that this graph is bipartite of diameter $n-1$ with minimum degree and connectivity also equal to $n-1$. These authors also studied whether or not these graphs contain Hamiltonian paths or cycles. 

We use $\chi(G)$ to denote the chromatic number of a graph $G$. Fabila-Monroy, Flores-Penaloza, Huemer, Hurtado, Urrutia and Wood \cite{FFHHUW} observed that $\chi(\mathcal{M}(K_{n,n}))=2$ and using this result, proved that
\begin{equation}\label{upperbnd}
\chi(\mathcal{M}(K_{2n}))\leq 4n-4.
\end{equation}
In \cite{FFHHUW}, these authors also made the following conjecture.
\begin{conj}[Fabila-Monroy, Flores-Penaloza, Huemer, Hurtado, Urrutia and Wood \cite{FFHHUW}]\label{chiconj}
For $n\geq 2$, 
\begin{equation}\label{chiconj:eq}
\chi(\mathcal{M}(K_{2n}))=n+1.
\end{equation}
\end{conj}
These authors confirmed their conjecture for $n\in \{2,3,4\}$ with the aid of a computer. We confirm these results by theoretical means in Section \ref{sec:chi}, and with extensive computation, we find a proper $6$-coloring of the graph $\mathcal{M}(K_{10})$ and a proper $7$-coloring of the graph $\mathcal{M}(K_{12})$\footnote{These colorings were obtained by computer and are available online at {\tt https://github.com/tanzkfp/FlipGraphsOnPerfectMatchings}}. We also obtain the following theoretical improvement of \eqref{upperbnd} in Section \ref{sec:chi}.
\begin{theorem}\label{ourupperbnd}
Let $n\geq 3$. If $q$ is the smallest prime power such that $q\geq 2n+1$, then 
$$\chi(\mathcal{M}(K_{2n}))\leq q.
$$
\end{theorem}
In Section \ref{sec:partitions}, we investigate the structure of the graph $\mathcal{M}(K_{2n})$.  
Given a perfect matching $M$ of $K_{2n}$, we show that the graph induced by the perfect matchings at distance $n-1$ from $M$ is isomorphic to the signed reversal graph $\sr{n-1}$ on $n-1$ symbols. This graph has been well studied in discrete mathematics and molecular biology \cite{B,HP} and is related to pancake graphs \cite{GP}, burnt pancake graphs \cite{BC} and reversal graphs \cite{CT}. 

For a positive integer $k$, let $S_k$ denote the set of permutations of the set $\{1,\dots,k\}$, where a permutation $\sigma \in S_k$ is represented by the tuple $(\sigma(1), \sigma(2), \ldots, \sigma(k))$.
A \emph{signed permutation} of degree $k$ is a $k$-tuple $(\sigma_1,\dots,\sigma_k)$ of integers such that the $k$-tuple $(|\sigma_1|,\dots,|\sigma_k|)$ of absolute values is a permutation in $S_k$. In other words, a signed permutation is obtained from a permutation by negating some entries of the corresponding tuple. Let $S_k^{\pm}$ denote the set of all signed permutations of degree $k$. The \emph{signed reversal graph} $\sr{k}$ is the graph whose vertices are the signed permutations of degree $k$ where $(\sigma_1,\dots,\sigma_k)$ is adjacent to $(\tau_1,\dots,\tau_k)$ if there exist $1\leq i\leq j\leq k$ such that 
\begin{equation}\label{sradj}
(\tau_1,\dots,\tau_{i-1},\tau_i,\dots,\tau_j,\tau_{j+1},\dots,\tau_k)=(\sigma_1,\dots,\sigma_{i-1},-\sigma_j,\dots,-\sigma_i,\sigma_{j+1},\dots,\sigma_k).
\end{equation}

For example, when $k=5$, the vertex $(2,-3,1,4,-5)$  is adjacent to $(-2,-3,1,4,-5)$ (take $i=j=1$ in \eqref{sradj}) and also to $(2,-3,1,5,-4)$ (take $i=4$, $j=5$). For simplicity of notation, we will denote the sign of each entry as an exponent so $(2,-3,1,4,-5)$ is the same as $(2^{+},3^{-},1^{+},4^{+},5^{-})$ or just $2^{+}3^{-}1^{+}4^{+}5^{-}$. The graph $\sr{k}$ has $2^k\cdot k!$ vertices and is regular of valency ${k+1\choose 2}$. In Section \ref{sec:partitions}, we describe the connections between the flip graph $\mathcal{M}(K_{2n})$ and the signed reversal graph $\sr{n-1}$. In particular, we show that $\mathcal{M}(K_{2n})$ has an equitable partition in which the subgraph induced by each cell is the disjoint union of isomorphic signed reversal graphs or the disjoint union of isomorphic Cartesian products of signed reversal graphs. We use this partition to show that $\chi(\mathcal{M}(K_{2n}))\leq \chi(\sr{n-1})+\chi(\sr{n-2})$. In Section \ref{sec:signreversal}, we determine some of the eigenvalues of the signed reversal graphs.
We finish the paper with some open problems in Section \ref{sec:open}.

\section{The eigenvalues of $\mathcal{M}(K_{2n})$}\label{sec:eigflip}

Let $M$ be a perfect matching of $K_{2n}$. The subgroup of the symmetric group $S_{2n}$ fixing $M$ is denoted by $H_n$ and is known as the hyperoctahedral group of degree $n$. It is isomorphic to the wreath product $S_2 \wr S_n$ and has order $2^n\cdot n!$. The graph $\mathcal{M}(K_{2n})$ can be identified with the quotient $S_{2n}/H_n$ and is a part of the perfect matching association scheme (see Godsil and Meagher \cite[Section 15.4]{GM}). This is useful due to the decomposition of $\mathcal{L}(\mathcal{M}(K_{2n}))=\{f:V(\mathcal{M}(K_{2n}))\rightarrow \mathbb{C}\}$ into irreducible representations.
\begin{theorem}[see Saxl \cite{saxl_1981} or Thrall \cite{Thrall}]
If $\mathcal{M}(K_{2n}) = S_{2n}/H_n$, then
\begin{equation}
\mathcal{L}(\mathcal{M}(K_{2n}))\cong \bigoplus_{\lambda\, \vdash\,n} S^{2\lambda},
\end{equation} 
where the direct sum is over all partitions $\lambda$ of $n$, $2\lambda = (2\lambda_1, 2\lambda_2,\cdots, 2\lambda_k)$ and $S^{2\lambda}$ is the associated irreducible representation of the symmetric group $S_{2n}$.
\end{theorem}
This theorem has been used to obtain the spectrum of the graph $\mathcal{M}(K_{2n})$.
\begin{theorem}[see Diaconis-Holmes \cite{DH2} or Chapter~7 in MacDonald \cite{MacDonald}]
The graph $\mathcal{M}(K_{2n})$ has an eigenvalue $\beta_\lambda$ for each partition $\lambda=(\lambda_1,\lambda_2,\ldots,\lambda_k)$ of $n$, given by
\begin{equation}\label{betalambda}
\beta_\lambda = \sum_{j=1}^k \lambda_j(\lambda_j-j).
\end{equation}
The multiplicity of $\beta_\lambda$ is determined by the partition $\mu=2\lambda$, and is given by
$$\text{mult}(\beta_{\lambda})=\frac{(2n)!}{\prod_{(i,j)\in\mu}h(i,j)},$$
with the product being over the cells of the Young diagram for $\mu$, and the hook length $h(i,j)=\mu_i+\mu_j'-i-j+1$, where $\mu'$ is the transposed diagram.
\end{theorem}
Note that different partitions $\lambda$ may give the same eigenvalue $\beta_{\lambda}$ and in such situations, the multiplicity of that eigenvalue would be the sum of the given multiplicities, taken over all the $\lambda$ that produce that eigenvalue.

The smallest eigenvalue can be determined easily.
\begin{corollary}
The smallest eigenvalue of the graph $\mathcal{M}(K_{2n})$ is $-{n\choose 2}$. 
\end{corollary}
\begin{proof}
Since the eigenvalues $\beta_{\lambda}$ described by \eqref{betalambda} are increasing with respect to the majorization order on partitions (see \cite[p.382]{CST} for a proof) the smallest eigenvalue corresponds to the partition $(1^n)$ and equals
\begin{equation*}
\beta_{(1^n)}=\sum_{j=1}^n1(1-j)=-{n\choose 2}.
\end{equation*}
\end{proof}

If $G$ is an undirected non-empty graph whose adjacency matrix eigenvalues are $\theta_1\geq \dots \geq \theta_{\min}$, 
then Hoffman \cite{Hoffman} (see also \cite[Theorem 3.6.2]{BH}) proved that
\begin{equation}\label{hoff}
\chi(G)\geq 1+\frac{\theta_1}{|\theta_{\min}|}.
\end{equation}
Applying this bound for $\mathcal{M}(K_{2n})$, we get that $\chi(\mathcal{M}(K_{2n}))\geq 1+\frac{n(n-1)}{{n\choose 2}}=3$. This is certainly true, although not as strong as a bound as we were hoping for. This lower bound can be also deduced combinatorially since the subgraph induced by the neighborhood of any vertex of $\mathcal{M}(K_{2n})$ is a perfect matching with ${n\choose 2}$ edges. Note that this is worse than the lower bound 
$$\chi(\mathcal{M}(K_{2n}))\geq \chi(\mathcal{M}(K_{8}))=5,$$
for any $n\geq 4$, which is obtained from the observation that $\chi(\mathcal{M}(K_{2n}))$ is non-decreasing with $n$ (as $\mathcal{M}(K_{2n-2})$ is an induced subgraph of $\mathcal{M}(K_{2n})$).

The Lov\'{a}sz theta function of the complement of a graph $\overline{H}$ provides a stronger lower bound for the chromatic number of $H$ than Hoffman's ratio bound \eqref{hoff} (see \cite[Theorem 6]{Lovasz} and the Sandwich theorem \cite{Knuth}). Our computations for small values of $n$ seem to indicate that this lower bound also equals $3$ for $\mathcal{M}(K_{2n})$.

\section{Some structural properties of $\mathcal{M}(K_{2n})$}\label{sec:partitions}

In this section, we investigate the structure of $\mathcal{M}(K_{2n})$ and  explain the connection with signed reversal graphs. This will be useful in providing upper bounds for the chromatic number of $\mathcal{M}(K_{2n})$ for small $n$. We label the vertices of $K_{2n}$ by $0^{+},0^{-},1^{+},1^{-},\dots,(n-1)^{+},(n-1)^{-}$, and distinguish a particular perfect matching, namely
\[
M_0 = \{0^{+},0^{-}\},\dots,\{(n-1)^{+},(n-1)^{-}\},
\]
as the \emph{identity perfect matching}.

If $M$ is an arbitrary perfect matching of $K_{2n}$ then the multigraph union $M_0 \cup M$ is the disjoint union of cycles of even lengths, say $2\lambda_1\geq 2\lambda_2\geq \dots \geq 2\lambda_k$ (for some $k$), whose lengths add up to $2n$. Dividing by $2$ yields a partition $(\lambda_1,\cdots,\lambda_k)\vdash n$, which we call the \emph{type} of the matching $M$. Note that $(1^n)$ is the type of the identity perfect matching. 

\begin{figure}[h]
    \centering

    
\begin{minipage}[b]{0.25\textwidth}
\begin{center}
\begin{tikzpicture}[edge_style/.style={color=blue, line width=1.3pt}, identity_style/.style= {color=red, line width=1.3pt}, vertex/.style = {circle,fill=black,minimum size=5pt,inner sep=1pt}] 
   \node[vertex,label=left:$0^{+}$] (G_1) at (0,1){} ;
  \node[vertex,label=left:$1^{+}$] (G_2) at (0,0) {} ;
  \node[vertex,label=left:$2^{+}$] (G_3) at (0,-1){};
  \node[vertex,label=right:$2^-$] (G_4) at (1,-1){};
  \node[vertex,label=right:$1^-$] (G_5) at (1,0){};
  \node[vertex,label=right:$0^-$] (G_6) at (1,1){};
  \draw[identity_style] (G_1) edge (G_6);
  \draw[identity_style] (G_5) edge (G_2);
  \draw[identity_style] (G_3) edge (G_4);
 \draw[edge_style]  (G_1) edge (G_2);
 \draw[edge_style]  (G_5) edge (G_4);
 \draw[edge_style]  (G_6) edge (G_3);
\end{tikzpicture}
\end{center}
\caption{Label $(1^{-}2^{+})$ Type $(3)$}
\label{examples1}
\end{minipage}
 \hspace{1cm}
\begin{minipage}[b]{0.25\textwidth}
\begin{center}
\begin{tikzpicture}[edge_style/.style={color=blue, line width=1.3pt}, identity_style/.style= {color=red, line width=1.3pt}, vertex/.style = {circle,fill=black,minimum size=5pt,inner sep=1pt}] 
  \node[vertex,label=left:$0^{+}$] (G_1) at (0,1){} ;
  \node[vertex,label=left:$1^{+}$] (G_2) at (0,0) {} ;
  \node[vertex,label=left:$2^{+}$] (G_3) at (0,-1){};
  \node[vertex,label=right:$2^-$] (G_4) at (1,-1){};
  \node[vertex,label=right:$1^-$] (G_5) at (1,0){};
  \node[vertex,label=right:$0^-$] (G_6) at (1,1){};
  \draw[edge_style]  (G_1) edge[bend right] (G_6);
 \draw[edge_style]  (G_2) edge (G_4);
 \draw[edge_style]  (G_5) edge (G_3);
   \draw[identity_style] (G_1) edge[bend left] (G_6);
  \draw[identity_style] (G_5) edge (G_2);
  \draw[identity_style] (G_3) edge (G_4);

\end{tikzpicture}\end{center}
    \caption{Label $()(2^{+})$ Type $(2,1)$}
    \label{examples2}
    \end{minipage}
\end{figure}

In this multigraph $M_0 \cup M$ color the edges from the identity matching in red, and those from the matching $M$ in blue. Now consider a walk starting from $0^+$ and alternating blue and red edges until the walk returns to $0^+$, noting down the sequence formed by the vertices at the beginning of each blue edge (other than the first). If there are cycles that have not yet been traversed, repeat this process by starting with the next smallest unused positive number. We note that by fixing the start point of each cycle, this representation is unique. Examples of labels and types for two matchings are shown in Figure \ref{examples1} and \ref{examples2}.

Given a graph $H=(V,E)$, a partition $V=X_1\cup \dots \cup X_t$ of its vertex set is called \textit{equitable} if there exist non-negative integers $b_{i,j}$ for $1\leq i,j\leq t$ such that for any $1\leq i,j\leq t$ and for any vertex $x\in X_i$, the number of neighbors of $x$ that are contained in $X_j$ equals $b_{i,j}$. The $t\times t$ matrix $B=(b_{i,j})_{1\leq i,j\leq t}$ is called the \emph{quotient matrix} of the partition. In general, $B$ is not symmetric, but it is always diagonalizable, and its spectrum is contained in the spectrum of the adjacency matrix of $H$ (see \cite[Section 2.3]{BH} or \cite[Chapter 5]{Godsil} for example). It is well-known that if $\Gamma$ is a group of automorphisms of $H$, then the orbits of $\Gamma$ form an equitable partition of $V(H)$ (see \cite[p.76]{Godsil}).
\begin{prop}
The partition of the vertices of $\mathcal{M}(K_{2n})$ according to their type is equitable.
\end{prop}
\begin{proof}
The subgroup of $S_{2n}$ fixing the identity perfect matching is a subgroup of the automorphism group of $\mathcal{M}(K_{2n})$. The partition of the matchings according to their types is the partition of the vertex set of $\mathcal{M}(K_{2n})$ into the orbits of this subgroup  (see also \cite[Chapter 11]{CST}) which is equitable by the previous paragraph.
\end{proof}

\begin{prop}\label{3types}
If $M$ is a perfect matching of $K_{2n}$ with type $(\lambda_1,\dots,\lambda_k)$ and $M'$ is a neighbor of $M$, then the type of $M'$ is one of the following:
\begin{enumerate}
\item $(\lambda_1,\dots,\lambda_k)$,
\item $(\mu_1,\dots,\mu_{k-1})$ which is obtained from $(\lambda_1,\dots,\lambda_k)$ by combining two parts $\lambda_i$ and $\lambda_j$ into one part $\lambda_i+\lambda_j$ and leaving the remaining parts unchanged,
\item $(\nu_1,\dots,\nu_{k},\nu_{k+1})$ which is obtained from $(\lambda_1,\dots,\lambda_k)$ by breaking one part $\lambda_{\ell}$ into two smaller parts $\lambda'_{\ell}$ and $\lambda''_{\ell}$ and leaving the remaining parts unchanged.
\end{enumerate}
\end{prop}
\begin{proof}
The multigraph $M_0 \cup M$ is a disjoint union of even cycles. When two edges from $M$ are flipped to form $M'$ then the number of cycles of $M_0 \cup M'$ either stays the same, increases by one (if a cycle breaks into two), or decreases by one (if two cycles are merged). These three possibilities give the three possible types of $M'$ listed above.
\end{proof}

\begin{prop}\label{nSRn1}
The subgraph of $\mathcal{M}(K_{2n})$ induced by the vertices with type $(n)$ is isomorphic to the signed reversal graph $\sr{n-1}$.
\end{prop}
\begin{proof}
Let $H$ be the subgraph of $\mathcal{M}(K_{2n})$ induced by the vertices of type $(n)$. Consider a vertex/matching $M$ in $H$. Its union with the identity perfect matching must be a Hamiltonian cycle of $K_{2n}$ of the form:
$$
0^{-}0^{+}\alpha_1^{\epsilon_1}\alpha_1^{\overline{\epsilon_1}}\alpha_2^{\epsilon_2}\alpha_2^{\overline{\epsilon_2}}\cdots \alpha_{n-1}^{\epsilon_{n-1}}\alpha_{n-1}^{\overline{\epsilon_{n-1}}},
$$
where $(\alpha_1,\dots,\alpha_{n-1})$ is a permutation of the set $\{1,\dots,n-1\}$, $\epsilon_1,\dots,\epsilon_{n-1}\in \{+,-\}$, where we use the notations $\overline{+}=-$ and $\overline{-}=+$. The label of $M$ is $\alpha_1^{\overline{\epsilon_1}}\dots\alpha_{n-1}^{\overline{\epsilon_{n-1}}}$. We claim that the correspondence between the vertex set of $H$ and the vertex set of $\sr{n-1}$ given by the label function is a graph isomorphism. To see this, consider a vertex $M'$ of $H$ that is adjacent to $M$. Assume that $M'$ is obtained by flipping two edges of $M$, say $\alpha_i^{\overline{\epsilon_i}}\alpha_{i+1}^{\epsilon_{i+1}}$ and $\alpha_j^{\epsilon_j}\alpha_{j+1}^{\overline{\epsilon_{j+1}}}$ for some $1\leq i<j\leq n$. Then the label of $M'$ is 
$$\alpha_1^{\overline{\epsilon_1}}\cdots \alpha_i^{\overline{\epsilon_i}}\alpha_j^{\epsilon_j}\cdots\alpha_{i+1}^{\epsilon_{i+1}}\alpha_{j+1}^{\epsilon_{j+1}}\cdots \alpha_{n-1}^{\epsilon_{n-1}}.$$
This shows that edges of $H$ are mapped to edges in $\sr{n-1}$ by the label function. It is not hard to see that this correspondence also maps non-edges to non-edges and is actually an isomorphism.
\end{proof}

Let  $G=(V,E)$ and $H=(W,F)$ be two graphs. The \textit{Cartesian or box product} $G\cart H$ has vertex set $V\times W$ where $(g_1,h_1)\sim (g_2,h_2)$ if $g_1\sim g_2$ in $G$ and $h_1=h_2$, or if $g_1=g_2$ and $h_1\sim h_2$ in $H$. This definition can be extended by associativity to the Cartesian product of more than two graphs. The following result is well known (see Sabidussi \cite[Lemma 2.6]{Sabidussi} for $k=2$).
\begin{prop}\label{sabidussichrom}
If $H_1,\dots,H_k$ are graphs, then
\begin{equation}
\chi(H_1\cart \cdots \cart H_k)=\max_{j:1\leq j\leq k}\chi(H_j).
\end{equation}
\end{prop}

For the following proposition, we extend the definition of $\sr{n}$ to $n=0$ by defining the signed reversal graph $\sr0$ to be the graph with one vertex and no edges. Note that $\sr1$ is isomorphic to the complete graph $K_2$ and $\sr2$ is isomorphic to the $3$-dimensional cube as shown in Figure \ref{sr2}.

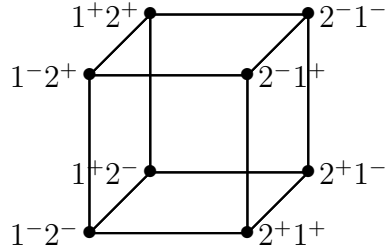
\begin{figure}[h!]
\centering
\begin{tikzpicture}[scale=0.7]
\newcommand{\Depth}{3}
\newcommand{\Height}{3}
\newcommand{\Width}{3}
\coordinate[label=left:$1^{+}2^{-}$] (O) at (0,0,0);
\coordinate[label=left:$1^{+}2^{+}$] (A) at (0,\Width,0);
\coordinate[label=left:$1^{-}2^{+}$] (B) at (0,\Width,\Height);
\coordinate[label=left:$1^{-}2^{-}$] (C) at (0,0,\Height);
\coordinate[label=right:$2^{+}1^{-}$] (D) at (\Depth,0,0);
\coordinate[label=right:$2^{-}1^{-}$] (E) at (\Depth,\Width,0);
\coordinate[label=right:$2^{-}1^{+}$] (F) at (\Depth,\Width,\Height);
\coordinate[label=right:$2^{+}1^{+}$] (G) at (\Depth,0,\Height);

\draw [thick] (O) -- (C) -- (G) -- (D) -- cycle;
\draw [thick](O) -- (A) -- (E) -- (D) -- cycle;
\draw [thick](O) -- (A) -- (B) -- (C) -- cycle;
\draw [thick](D) -- (E) -- (F) -- (G) -- cycle;
\draw [thick](C) -- (B) -- (F) -- (G) -- cycle;
\draw [thick](A) -- (B) -- (F) -- (E) -- cycle;
\draw (O) node{\textbullet};
\draw (A) node{\textbullet};
\draw (B) node{\textbullet};
\draw (C) node{\textbullet};
\draw (D) node{\textbullet};
\draw (E) node{\textbullet};
\draw (F) node{\textbullet};
\draw (G) node{\textbullet};

\end{tikzpicture}
\caption{The signed reversal graph $\sr2$}
\label{sr2}
\end{figure}

Denote by $\mathcal{M}_\lambda$ the subgraph of $\mathcal{M}(K_{2n})$ induced by the vertices of type $\lambda = (\lambda_1,\dots,\lambda_k)$. 
\begin{prop}
Each connected component of the  subgraph of $\mathcal{M}(K_{2n})$ induced by the vertices with type $(\lambda_1,\dots,\lambda_k)$ is isomorphic to the graph $\sr{\lambda_1-1}\cart \cdots \cart \sr{\lambda_{k}-1}$. 
\end{prop}
\begin{proof}
Given a perfect matching $M$, let $\pi_M$ denote the set partition of $V(K_{2n})$ induced by the cycles of $M \cup M_0$. 

If $M$ and $M'$ are related by a flip-move and they both have the same type, then $\pi_M = \pi_{M'}$, because the two edges to be flipped must both be chosen from the same cycle of $M \cup M_0$.  In particular, if $\pi_M \not= \pi_{M'}$ then the two matchings are not in the same connected component of $\mathcal{M}_\lambda$.

Now, for any fixed set partition $\pi$ of $V(K_{2n})$ with cells of size $2\lambda_1$, $2\lambda_2$, $\ldots$, $2\lambda_k$ consider all the perfect matchings $M$ in $\mathcal{M}_\lambda$ such that $\pi_M = \pi$. Then arguments similar to Proposition \ref{3types} and Proposition \ref{nSRn1} show that subgraph of $\mathcal{M}_\lambda$ induced by these vertices is isomorphic to $\sr{\lambda_1-1}\cart \cdots \cart \sr{\lambda_{k}-1}$. 
\end{proof}

The next result is an immediate corollary of Proposition~\ref{sabidussichrom}. 
\begin{corollary}
If $\lambda = (\lambda_1, \ldots, \lambda_k)$, then the chromatic number of the subgraph $\mathcal{M}_\lambda$ is $\chi(\sr{\lambda_1-1})$.
\end{corollary}

\begin{prop}\label{upperbndsrn}
Let $n\geq 3$ be an integer. Then $\chi(\mathcal{M}(K_{2n}))\leq \chi(\sr{n-1})+\chi(\sr{n-2})$.
\end{prop}
\begin{proof}
Consider the partition of the vertices of $\mathcal{M}(K_{2n})$ according to their type. Now form a coarser partition $V_1$, $V_2$, $\ldots$, $V_n$ where 
\[
V_\ell = \bigcup_{|\lambda|=\ell} \mathcal{M}_\lambda.
\]
In other words, $V_\ell$ contains all the perfect matchings $M$ such that $M \cup M_0$ has $\ell$ cycles. A perfect matching in $V_\ell$ has neighbours only in $V_{\ell-1}$, $V_\ell$ and $V_{\ell+1}$.
Now the subgraph induced by $V_1$ is isomorphic to $\sr{n-1}$ and so (obviously) can be colored in $\chi(\sr{n-1})$ colors. Every connected component of the subgraph induced by $V_2$ has the form $\sr{t-1} \cart \sr{n-t-1}$ for some $1\leq t\leq n-1$ and so can be colored with $\chi(\sr{n-2})$ colors, with these colors chosen to be distinct from the colors used on $V_1$.
The ``odd layers'' $V_3$, $V_5$, $V_7$, $\ldots$ can be colored using colors from those used to color $V_1$ and the ``even layers'' $V_4$, $V_6$, $V_8$, $\ldots$ can be colored using colors from those used to color $V_2$. In total, at most $\chi(\sr{n-1})+\chi(\sr{n-2})$ colors are required.\end{proof}


We conclude this section with an observation regarding the structure of the signed reversal graph $\sr{k}$ for $k\geq 2$. Let $\pi =(\pi_1,\dots, \pi_k)$ be a permutation in $S_k$, and denote by $V_{\pi}$ the set of $2^k$ signed permutations of the form $(\pi_1^{\epsilon_1}, \dots, \pi_k^{\epsilon_k})$, where $\epsilon_1,\dots,\epsilon_k\in \{+,-\}$. A vertex of $\sr{k}$ lying in $V_\pi$ is adjacent to exactly $k$ other vertices in $V_\pi$, namely the $k$ signed permutations obtained by changing the sign in exactly one coordinate. If $\pi'$ is a permutation obtained from $\pi$ by substring reversal, then each vertex of $V_\pi$ is adjacent to a unique vertex in $V_{\pi'}$. In this case, there is a matching between $V_\pi$ and $V_{\pi'}$. The \emph{reversal graph} $R_k$ is the graph with vertex set $S_k$, where two permutations are adjacent if they are substring reversals of each other. See Section \ref{sec:signreversal} and Chung and Tobin \cite{CT} for more details on these graphs.
\begin{prop}\label{equitablesrn1}
Let $k\geq 2$ be an integer. The partition of the signed reversal graph $\sr{k}$ into cells $\{V_{\pi} : \pi\in S_k\}$ is an equitable partition with quotient matrix  $kI_{k!}+A(R_k)$.
\end{prop}
\begin{proof}
Note that the subgraph of $\sr{k}$ induced by $V_{\pi}$ is isomorphic to the $k$-dimensional cube for any $\pi \in S_k$. For any permutations $\pi\neq \pi'\in S_k$, each vertex in $V_{\pi}$ has exactly one neighbor in $V_{\pi'}$ if $\pi$ and $\pi'$ are adjacent in the reversal graph $R_k$ and no neighbors, otherwise. 
\end{proof}

\section{Chromatic numbers}\label{sec:chi}

In this section we consider bounds on the chromatic number of $\mathcal{M}(K_{2n})$. In order to apply  Proposition~\ref{upperbndsrn} we need to know the chromatic number of the signed reversal graph $\sr{n}$, and so we start there.

\begin{figure}[!h]
\begin{center}
\begin{tikzpicture}[scale=0.9,node_style/.style={circle,draw=black}, edge_style/.style= {color=blue, line width=1.3pt},  identity_style/.style= {color=red, line width=1.3pt}] 

     \node[node_style] (v1) at (0,2) {123};
     \node[node_style] (v2) at (1,1) {213};
     \node[node_style] (v3) at (1,-1) {231};
     \node[node_style] (v4) at (0,-2) {321};
     \node[node_style] (v5) at (-1,-1) {312};
     \node[node_style] (v6) at (-1,1) {132};
     \draw[edge_style]  (v1) edge (v2);
     \draw[edge_style]  (v2) edge (v3);
     \draw[edge_style]  (v3) edge (v4);
     \draw[edge_style]  (v4) edge (v5);
     \draw[edge_style]  (v5) edge (v6);
     \draw[edge_style]  (v6) edge (v1);
     \draw[identity_style]  (v4) edge (v1);
     \draw[identity_style]  (v5) edge (v2);
     \draw[identity_style]  (v3) edge (v6);
\end{tikzpicture}    
\hspace{1cm}
\begin{tikzpicture}[scale=0.9,node_style/.style={circle,draw=black}, edge_style/.style= {color=blue, line width=1.3pt},  identity_style/.style= {color=red, line width=1.3pt}] 
     \node[node_style] (v1) at (0,2) {12};
     \node[node_style] (v2) at (1,1) {23};
     \node[node_style] (v3) at (1,-1) {31};
     \node[node_style] (v4) at (0,-2) {12};
     \node[node_style] (v5) at (-1,-1) {23};
     \node[node_style] (v6) at (-1,1) {31};
     \draw[edge_style]  (v1) edge (v2);
     \draw[edge_style]  (v2) edge (v3);
     \draw[edge_style]  (v3) edge (v4);
     \draw[edge_style]  (v4) edge (v5);
     \draw[edge_style]  (v5) edge (v6);
     \draw[edge_style]  (v6) edge (v1);
     \draw[identity_style]  (v4) edge (v1);
     \draw[identity_style]  (v5) edge (v2);
     \draw[identity_style]  (v3) edge (v6);
     \end{tikzpicture}    
\end{center}
\caption{Equitable partition (left) and a proper $3$-coloring (right) of $\sr3$}
\label{fig:sr3}
\end{figure}
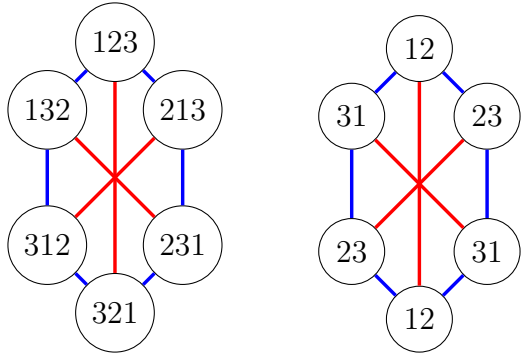

\subsection{The chromatic number of $\sr{n}$}

As $\sr{1}$ is a single edge, and $\sr2$ is the cube, these graphs both have chromatic number $2$. The graph $\sr3$ has an equitable partition into $3!=6$ cells of the form $V_{\pi}$ for $\pi \in S_3$,  where the graph induced on each cell is the $3$-cube. The first diagram of Figure \ref{fig:sr3} shows how the cells are connected, with blue edges indicating two cells connected by a matching induced by a substring reversal of length $2$ and red edges indicating cells connected by a substring reversal of length $3$ (i.e., just reversing the entire permutation).  As each cell induces a cube, which is bipartite, we intend to find a coloring that uses just two colors on each cell. Between cells $V_\pi$ and $V_{\pi'}$ there are either no edges, or a matching induced by the same signed substring reversal on every element of $V_\pi$. Each cell contains a unique (unsigned) permutation and so we can specify the coloring of a cell by the two colors used on the subgraph induced by $V_\pi$ with the convention that we use the first color on the vertex of $V_{\pi}$ that has $+$ on all its entries. The right-hand diagram of Figure \ref{fig:sr3} describes a coloring of the cells $V_{\pi}$. The label $12$ on $V_{123}$ indicates that the $3$-dimensional cube induced by $V_{123}$ is colored with colors $1$ and $2$, with vertex $1^+2^+3^+$ colored $1$. The neighborhood of $1^+2^+3^+$ consists of three vertices in $V_{123}$ (each of them colored $2$), the vertex $1^+3^-2^-$ in $V_{132}$ (which has color $3$), the vertex $2^-1^-3^+$ in $V_{213}$ (which is colored 2) and $3^-2^-1^-$ in $V_{321}$ (which has color $2$). One can verify that this is a proper $3$-coloring, and hence $\chi(\sr{3}) \leq 3$. The graph $\sr{3}$ contains a cycle on $7$ vertices: 
$$
1^+2^+3^+\sim 1^+2^+3^-\sim 1^+2^-3^- \sim 1^-2^-3^-\sim 2^+1^+3^- \sim 2^+3^+1^-\sim 3^-2^-1^- \sim 1^+2^+3^+,
$$
and therefore $\chi(\sr{3})=3$.



A similar approach can be used to exhibit a $4$-colouring of $\sr4$. Figure \ref{fig:sr4} shows the equitable partition of $\sr4$ into $24$ cells, with each cell containing $16$ vertices inducing a $4$-cube. In this diagram, blue edges indicate that two cells are joined by a matching induced by a substring reversal of length $2$,  and the red edges indicate cells connected by a matching induced by a substring reversal of length $3$. The diagram shows the cells in ``layers'' so that blue edges connect cells in adjacent layers, while the red edges connect cells that are either one or three layers apart. Hence by using colors $\{1,2\}$ for each cell in the odd layers and $\{3,4\}$ for each cell in even layers, all of the edges within the cells are properly colored and all the blue and red matchings are properly colored. This only leaves the matchings between cells related by a substring reversal of length $4$ (i.e., the entire permutation is reversed). The colors chosen for each cell shown in Figure~\ref{fig:sr4} ensure that the matching between, say, $V_{2341}$ and $V_{1432}$ is properly colored, and similarly for all the other pairs of cells related by full reversals.
Unfortunately, we are not aware of a theoretical way to prove that $\chi(\sr{4})>3$ and we rely on the computer for this part.

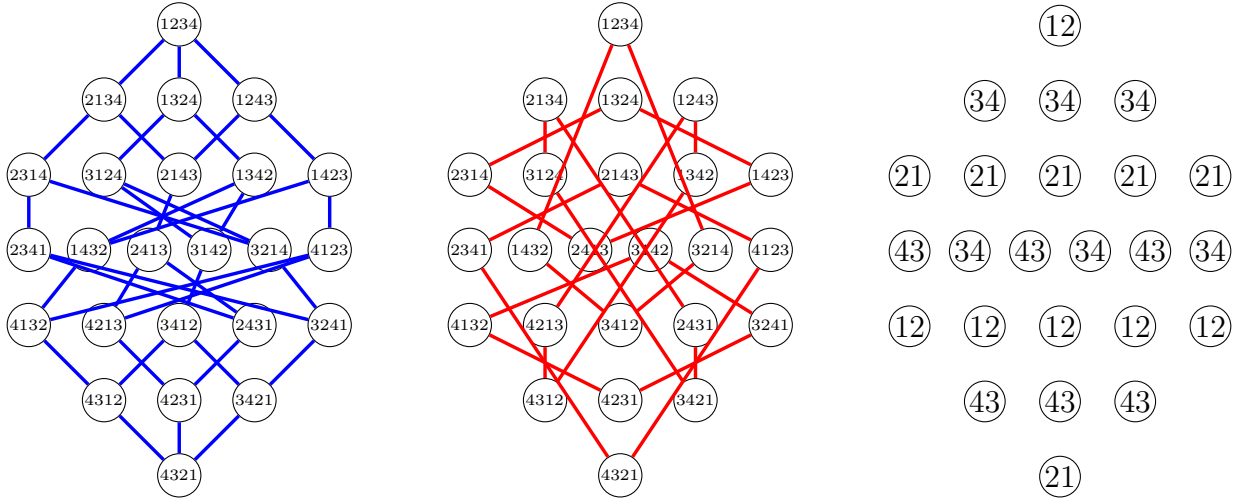
\begin{figure}[!ht]
\begin{center}
\begin{tikzpicture}[node_style/.style={circle,draw=black,inner sep=0.5pt,minimum size=0pt}, edge_style/.style= {color=blue, line width=1.3pt}] 

     \node[node_style] (1234) at (0,3)   {\tiny 1234};
     \node[node_style] (2134) at (-1,2)   {\tiny 2134};
     \node[node_style] (1243) at (1,2)  {\tiny 1243};
     \node[node_style] (1324) at (0,2)   {\tiny 1324};

     \node[node_style] (1342) at (1,1)   {\tiny 1342};
     \node[node_style] (2143) at (0,1)  {\tiny 2143};
     \node[node_style] (3124) at (-1,1)   {\tiny 3124};
     \node[node_style] (1423) at (2,1)  {\tiny 1423};
     \node[node_style] (2314) at (-2,1)  {\tiny 2314};

     \node[node_style] (2341) at (-2,0)  {\tiny 2341};
     \node[node_style] (4123) at (2,0)  {\tiny 4123};
     \node[node_style] (1432) at (-1.2,0)  {\tiny 1432};
     \node[node_style] (3214) at (1.2,0)   {\tiny 3214};
     \node[node_style] (3142) at (0.4,0)  {\tiny 3142};
     \node[node_style] (2413) at (-0.4,0) {\tiny 2413};

     \node[node_style] (4132) at (-2,-1) {\tiny 4132};
     \node[node_style] (3241) at (2,-1)  {\tiny 3241};
     \node[node_style] (4213) at (-1,-1)  {\tiny 4213};
     \node[node_style] (3412) at (0,-1)  {\tiny 3412};
     \node[node_style] (2431) at (1,-1)  {\tiny 2431};

     \node[node_style] (3421) at (1,-2)   {\tiny 3421};
     \node[node_style] (4312) at (-1,-2)  {\tiny 4312};
     \node[node_style] (4231) at (0,-2)  {\tiny 4231};
     \node[node_style] (4321) at (0,-3)   {\tiny 4321};

     \draw[edge_style]  (1234) edge (2134);
     \draw[edge_style]  (1234) edge (1324);
     \draw[edge_style]  (1234) edge (1243);
     \draw[edge_style]  (2314) edge (2134);
     \draw[edge_style]  (2143) edge (2134);     
     \draw[edge_style]  (3124) edge (1324);
     \draw[edge_style]  (1342) edge (1324);
     \draw[edge_style]  (2143) edge (1243);
     \draw[edge_style]  (1423) edge (1243);
     \draw[edge_style]  (2314) edge (2341);
     \draw[edge_style]  (2314) edge (3214);
     \draw[edge_style]  (2143) edge (2413);     
     \draw[edge_style]  (3124) edge (3214);
     \draw[edge_style]  (3124) edge (3142);
     \draw[edge_style]  (1342) edge (3142);
     \draw[edge_style]  (1342) edge (1432);
     \draw[edge_style]  (1423) edge (1432);
     \draw[edge_style]  (1423) edge (4123);
     \draw[edge_style]  (4132) edge (4312);
     \draw[edge_style]  (3412) edge (4312);     
     \draw[edge_style]  (4213) edge (4231);
     \draw[edge_style]  (2431) edge (4231);
     \draw[edge_style]  (3412) edge (3421);
     \draw[edge_style]  (3241) edge (3421);
     \draw[edge_style]  (3241) edge (2341);
     \draw[edge_style]  (3241) edge (3214);     
     \draw[edge_style]  (2431) edge (2341);
     \draw[edge_style]  (2431) edge (2413);
     \draw[edge_style]  (3412) edge (3142);
     \draw[edge_style]  (2413) edge (4213);
     \draw[edge_style]  (4123) edge (4213);
     \draw[edge_style]  (4132) edge (1432);
     \draw[edge_style]  (4132) edge (4123);
     \draw[edge_style]  (4321) edge (4312);
     \draw[edge_style]  (4321) edge (4231);
     \draw[edge_style]  (4321) edge (3421);
\end{tikzpicture}
\hspace{1cm}
\begin{tikzpicture}[node_style/.style={circle,draw=black,inner sep=0.5pt,minimum size=0pt}, edge_style/.style= {color=red, line width=1.3pt}] 
     \node[node_style] (1234) at (0,3)   {\tiny 1234};
     \node[node_style] (2134) at (-1,2)   {\tiny 2134};
     \node[node_style] (1243) at (1,2)  {\tiny 1243};
     \node[node_style] (1324) at (0,2)   {\tiny 1324};

     \node[node_style] (1342) at (1,1)   {\tiny 1342};
     \node[node_style] (2143) at (0,1)  {\tiny 2143};
     \node[node_style] (3124) at (-1,1)   {\tiny 3124};
     \node[node_style] (1423) at (2,1)  {\tiny 1423};
     \node[node_style] (2314) at (-2,1)  {\tiny 2314};

     \node[node_style] (2341) at (-2,0)  {\tiny 2341};
     \node[node_style] (4123) at (2,0)  {\tiny 4123};
     \node[node_style] (1432) at (-1.2,0)  {\tiny 1432};
     \node[node_style] (3214) at (1.2,0)   {\tiny 3214};
     \node[node_style] (3142) at (0.4,0)  {\tiny 3142};
     \node[node_style] (2413) at (-0.4,0) {\tiny 2413};

     \node[node_style] (4132) at (-2,-1) {\tiny 4132};
     \node[node_style] (3241) at (2,-1)  {\tiny 3241};
     \node[node_style] (4213) at (-1,-1)  {\tiny 4213};
     \node[node_style] (3412) at (0,-1)  {\tiny 3412};
     \node[node_style] (2431) at (1,-1)  {\tiny 2431};

     \node[node_style] (3421) at (1,-2)   {\tiny 3421};
     \node[node_style] (4312) at (-1,-2)  {\tiny 4312};
     \node[node_style] (4231) at (0,-2)  {\tiny 4231};
     \node[node_style] (4321) at (0,-3)   {\tiny 4321};

     \draw[edge_style]  (1234) edge (3214);
     \draw[edge_style]  (1234) edge (1432);
     \draw[edge_style]  (1432) edge (3412);
     \draw[edge_style]  (3214) edge (3412);
     
     \draw[edge_style]  (4321) edge (2341);
     \draw[edge_style]  (4321) edge (4123);
     \draw[edge_style]  (2341) edge (2143);
     \draw[edge_style]  (4123) edge (2143);
     
     \draw[edge_style]  (2314) edge (1324);
     \draw[edge_style]  (2314) edge (2413);
     \draw[edge_style]  (1324) edge (1423);
     \draw[edge_style]  (2413) edge (1423);

     \draw[edge_style]  (2134) edge (3124);
     \draw[edge_style]  (2134) edge (2431);
     \draw[edge_style]  (3124) edge (3421);
     \draw[edge_style]  (2431) edge (3421);

     \draw[edge_style]  (1243) edge (4213);
     \draw[edge_style]  (1243) edge (1342);
     \draw[edge_style]  (1342) edge (4312);
     \draw[edge_style]  (4213) edge (4312);
    
     \draw[edge_style]  (4231) edge (3241);
     \draw[edge_style]  (4231) edge (4132);
     \draw[edge_style]  (3241) edge (3142);
     \draw[edge_style]  (4132) edge (3142);    

\end{tikzpicture}
\hspace{1cm}
\begin{tikzpicture}[node_style/.style={circle,draw=black,inner sep=0.5pt,minimum size=0pt}, edge_style/.style= {color=blue, line width=1.3pt}] 

     \node[node_style] (1234) at (0,3)   { 12};
    
     \node[node_style] (2134) at (-1,2)   { 34};
     \node[node_style] (1243) at (1,2)  {34};
     \node[node_style] (1324) at (0,2)   { 34};

     \node[node_style] (1342) at (1,1)   {21};
     \node[node_style] (2143) at (0,1)  { 21};
     \node[node_style] (3124) at (-1,1)   {21};
     \node[node_style] (1423) at (2,1)  {21};
     \node[node_style] (2314) at (-2,1)  {21};

     \node[node_style] (2341) at (-2,0)  {43};
     \node[node_style] (4123) at (2,0)  {34};
     \node[node_style] (3214) at (1.2,0)  {43};
     \node[node_style] (1432) at (-1.2,0)   {34};
     \node[node_style] (3142) at (0.4,0)  {34};
     \node[node_style] (2413) at (-0.4,0) {43};

     \node[node_style] (3241) at (2,-1) {12};
     \node[node_style] (4132) at (-2,-1)  {12};
     \node[node_style] (2431) at (1,-1)  {12};
     \node[node_style] (3412) at (0,-1)  {12};
     \node[node_style] (4213) at (-1,-1)  {12};

     \node[node_style] (4312) at (1,-2)   {43};
     \node[node_style] (3421) at (-1,-2)  {43};
     \node[node_style] (4231) at (0,-2)  {43};

     \node[node_style] (4321) at (0,-3)   {21};
\end{tikzpicture}  

\end{center}

\caption{The equitable partition and a $4$-coloring for $\sr4$}
\label{fig:sr4}
\end{figure}

For $\sr5$, we can no longer describe the coloring ``by hand'', but by assuming that the cube induced by each cell $V_{\pi}$ is colored using only two colors, it is easily within computer range to verify that $\sr5$ has a $4$-coloring. For completeness, we present such a coloring in the Appendix. For $\sr{6}$ we can find a $5$-coloring by computer, but cannot even rule out the existence of a $4$-coloring.

\begin{table}
\begin{center}
\begin{tabular}{ccccccc}
\toprule
$n$ & 1 & 2 & 3 & 4 & 5 & 6\\
\midrule
$\chi(\sr{n})$ & 2& 2 & 3 & 4 & 4 & $\leq 5$\\
\bottomrule
\end{tabular}
\end{center}
\caption{Chromatic number of $\sr{n}$ for $n \leq 6$}
\label{chromaticsrn}
\end{table}

\subsection{The chromatic number of $\mathcal{M}(K_{2n})$}

For $n=2$, the graph $\mathcal{M}(K_4)$ is isomorphic to the complete graph on $3$ vertices which has chromatic number $3$. For $n=3$, the graph $\mathcal{M}(K_6)$ is the unique strongly regular graph with parameters $(15,6,1,3)$. This is the Kneser graph $K(6,2)$ which is well-known to have chromatic number $4$. (This is also a consequence of Proposition~\ref{upperbndsrn}.) For $n=4$, the graph $\mathcal{M}(K_8)$ has $105$ vertices, is $12$-regular. By Proposition~\ref{upperbndsrn} it has chromatic number number at most $\chi(\sr{2}) + \chi(\sr{3}) = 5$ and in fact its chromatic number is equal to $5$. For $n=5$, the graph $\mathcal{M}(K_{10})$ is a $20$-regular graph on $945$ vertices. Figure~\ref{flip5} shows its equitable partition into types and gives the number of neighbors of a vertex of one cell in each adjacent cell. 

By computation we have found both a proper $6$-coloring of $\mathcal{M}(K_{10})$ and a proper $7$-coloring of $\mathcal{M}(K_{12})$, thus providing supporting evidence for the conjecture of Fabila-Monroy et al. \cite{FFHHUW} that $\chi(\mathcal{M}(K_{2n})) = n+1$. In Table \ref{tab:chibounds}, we list the best lower and upper bounds for $\mathcal{M}(K_{2n})$ for $2\leq n\leq 11$. Note that the upper bounds $\mathcal{M}(K_{10})\leq 6$ and $\mathcal{M}(K_{12})\leq 7$ come from our computations and the upper bound $\mathcal{M}(K_{14})\leq 9$ comes from using Proposition \ref{upperbndsrn} and Table \ref{chromaticsrn}.

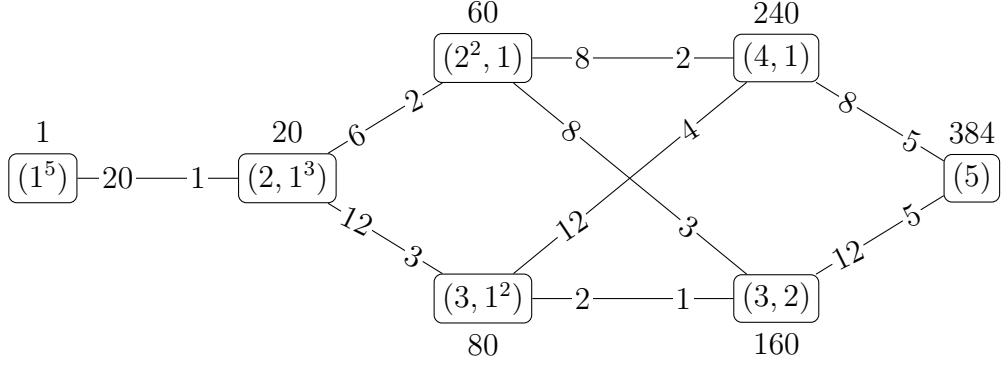
\begin{figure}[ht]
\begin{center}
\begin{tikzpicture}[xscale=1.3,yscale=1.6, node_style/.style={rectangle,rounded corners=3pt, draw=black,inner sep=3pt}, edge_style/.style= {color=black}] 
     \node[node_style] (v1) at (-0.5,0) { $\left(1^5\right)$};
     \node[node_style] (v2) at (2,0) { $\left(2,1^3\right)$};
     \node[node_style] (v3) at (4,1) { $\left(2^2,1\right)$};
     \node[node_style] (v4) at (4,-1){  $\left(3,1^2\right)$};
     \node[node_style] (v5) at (7,1) { $\left(4,1\right)$};
     \node[node_style] (v6) at (7,-1) { $\left(3,2\right)$};
     \node[node_style] (v7) at (9,0) { $\left(5\right)$};

     \draw[edge_style]  (v1) edge 
       node [inner sep=1pt, sloped, near start, fill=white] {20} 
       node [inner sep=1pt, sloped, near end, fill=white] {1} (v2);
     \draw[edge_style]  (v2) edge 
       node [inner sep=1pt, sloped, near start, fill=white] {6} 
       node [inner sep=1pt, sloped, near end, fill=white] {2} (v3);
     \draw[edge_style]  (v2) edge 
       node [inner sep=1pt, sloped, near start, fill=white] {12} 
       node [inner sep=1pt, sloped, near end, fill=white] {3}  (v4);
     \draw[edge_style]  (v4) edge 
       node [inner sep=1pt, sloped, near start, fill=white] {12} 
       node [inner sep=1pt, sloped, near end, fill=white] {4} (v5);
     \draw[edge_style]  (v3) edge 
       node [inner sep=1pt, sloped, near start, fill=white] {8} 
       node [inner sep=1pt, sloped, near end, fill=white] {2} (v5);
     \draw[edge_style]  (v4) edge 
            node [inner sep=1pt, sloped, near start, fill=white] {2} 
       node [inner sep=1pt, sloped, near end, fill=white] {1}  (v6);
     \draw[edge_style]  (v3) edge 
            node [inner sep=1pt, sloped, near start, fill=white] {8} 
       node [inner sep=1pt, sloped, near end, fill=white] {3} (v6);
     \draw[edge_style]  (v6) edge 
            node [inner sep=1pt, sloped, near start, fill=white] {12} 
       node [inner sep=1pt, sloped, near end, fill=white] {5} (v7);
     \draw[edge_style]  (v5) edge 
            node [inner sep=1pt, sloped, near start, fill=white] {8} 
       node [inner sep=1pt, sloped, near end, fill=white] {5} (v7);
     \node [above] at (v1.north) {$1$};
    \node [above] at (v2.north) {$20$};
    \node [above] at (v3.north) {$60$};
    \node [below] at (v4.south) {$80$};
    \node [above] at (v5.north) {$240$};
    \node [below] at (v6.south) {$160$};
    \node [above] at (v7.north) {$384$};
\end{tikzpicture}    
\end{center}
\caption{Equitable partition of $\mathcal{M}(K_{10})$ into types}
\label{flip5}
\end{figure}

\bigskip

For larger values of $n$, the best we can do is Theorem \ref{ourupperbnd}, which we now prove:
\begin{proof}[Proof Of Theorem~\ref{ourupperbnd}]
Let $q$ be the smallest prime power such that $q\geq 2n+1$, and let $\sigma: V(K_{2n}) \rightarrow \text{GF}(q)$ be an arbitrary injective function.  Then define a coloring of the flip graph as follows; if $X$ is a perfect matching with edges $\{x_1,x_2\},\dots,\{x_{2n-1},x_{2n}\}$, then color it with the color
\[
f(X):=\sigma(x_1)\sigma(x_2)+\sigma(x_3)\sigma(x_4)+\dots +\sigma(x_{2n-1})\sigma(x_{2n}).\
\]

Now we show that this coloring is a proper coloring. So let $Y$ be a matching with edges $\{y_1,y_2\},\dots,
\{y_{2n-1},y_{2n}\}$ that is adjacent to $X$. Without loss of generality we may assume that $\{x_{2i-1},x_{2i}\}=\{y_{2i-1},y_{2i}\}$ for any $i\geq 3$. If $\{y_1,y_2\}=\{x_1,x_3\}$ and $\{y_3,y_4\}=\{x_2,x_4\}$, then
\begin{align*}
f(Y)-f(X)&=\sigma(x_1)\sigma(x_3)+\sigma(x_2)\sigma(x_4)-\sigma(x_1)\sigma(x_2)-\sigma(x_3)\sigma(x_4)\\
&=\left(\sigma(x_1)-\sigma(x_4)\right)(\sigma(x_3)-\sigma(x_2)).
\end{align*}
If $\{y_1,y_2\}=\{x_1,x_4\}$ and $\{y_3,y_4\}=\{x_2,x_3\}$, then
\begin{align*}
f(Y)-f(X)&=\sigma(x_1)\sigma(x_4)+\sigma(x_2)\sigma(x_3)-\sigma(x_1)\sigma(x_2)-\sigma(x_3)\sigma(x_4)\\
&=\left(\sigma(x_1)-\sigma(x_3)\right) \left( \sigma(x_4)-\sigma(x_2)\right).
\end{align*}
As $\sigma$ is injective, the final value in each case is the product of non-zero values and so $f(Y)\neq f(X)$. Thus, we have a proper coloring of $\mathcal{M}(K_{2n})$ with $q$ colors and therefore, $\chi(\mathcal{M}(K_{2n}))\leq q$.
\end{proof}
Nagura \cite{Nagura} proved that for any $m\geq 25$, there is at least one prime between $m$ and $1.2m$. This means that for $n\geq 12$, there is a prime between $2n+1$ and $1.2(2n+1)=2.4n+1.2<4n-4$ and hence, our previous result improves the upper bound of $4n-4$ from \cite{FFHHUW}. It is straightforward to check it is also an improvement for $3\leq n\leq 11$ as seen in Table~\ref{tab:chibounds}. 

\begin{table}
\begin{center}
\begin{tabular}{ccccccccccc}
\toprule
$n$ & 2 & 3 & 4 & 5 & 6 & 7 & 8 & 9 & 10 & 11\\
\midrule
$4n-4$ & 4 & 8 & 12 & 16 & 20 & 24 & 28 & 32 & 36 & 40 \\
Th \ref{ourupperbnd} & 5 & 7 & 9 & 11 & 13 & 16 & 17 & 19 & 23 & 23\\
Upper & 3 & 4 & 5 & 6 & 7 & 9 & ? & ? & ? & ? \\
Lower & 3 & 4 & 5 & 5 & 5 & 5 & 5 & 5 & 5 & 5 \\
\bottomrule
\end{tabular}
\end{center}
\caption{Bounds on chromatic number of $\mathcal{M}(K_{2n})$}
\label{tab:chibounds}
\end{table}

Dusart \cite{Dusart} proved that for $m$ sufficiently large, there is a prime between $m$ and $m\left(1+\frac{1}{\ln m}\right)$ and therefore 
$$\chi(\mathcal{M}(K_{2n}))\leq (2n+1)\left(1+\frac{1}{\ln^3(2n+1)}\right)=(2n+1)(1+o(1)),$$ 
for $n$ sufficiently large.

\section{Spectral properties of the signed reversal graph}
\label{sec:signreversal}

For $n\geq 1$, let $X=X_n$ denote the $n\times n$ matrix whose $(i,j)$-th entry equals $\min(i,j,n-i+1,n-j+1)$. For a real number $x$, let $D$ be the unique diagonal matrix such that each row of $D+X$ sums to $x$. Chung and Tobin \cite[Lemma 9]{CT} proved that the eigenvalues of $D+X$ are $\mu_k=x-\lfloor \frac{k}{2}\rfloor n +2{\lfloor \frac{k}{2}\rfloor \choose 2}$ for $1\leq k\leq n$. Using this result, these authors showed that the second largest eigenvalue of the reversal graph $R_n$ is ${n\choose 2}-n$ (with the largest eigenvalue being ${n\choose 2}$). In this section, we use two equitable partitions of the signed reversal graph $\sr{n}$ to determine part of its spectrum.
\begin{prop}\label{eigsr}
The spectrum of $\sr{n}$ contains the following eigenvalues:
\begin{enumerate}
\item the eigenvalues of $D'+X$ and $D'-X$, where $X$ is the Chung-Tobin matrix from above and $D'$ is the diagonal degree matrix that makes each row sum of $D'+X$ equal to ${n+1\choose 2}$.
\item $\mu+n$, where $\mu$ is an eigenvalue of the reversal graph $R_n$.
\end{enumerate}
\end{prop}
\begin{proof}
For $1\leq j\leq n$, let $U_{j}(+)=\{\sigma: \sigma_j=n^{+}\}$ and $U_{j}(-)=\{\sigma:\sigma_j=n^{-}\}$. We claim that the partition of the vertex set of $\sr{n}$ into the $2n$ sets $U_{1}(+),\dots,U_{n}(+),U_{1}(-),\dots,U_{n}(-)$ is an equitable partition whose quotient matrix is 
\begin{equation}
\begin{bmatrix} D' & X\\
X& D'
\end{bmatrix},
\end{equation}
where $X=X_n$ is the Chung-Tobin matrix from the previous paragraph and $D'$ is the unique diagonal matrix such that $D'+X$ has each row sum ${n+1\choose 2}$.

To see this, note that for any $1\leq i\neq j\leq n$, there are no edges between $U_{i}(+)$ and $U_{j}(+)$ because moving $n^{+}$ from position $i$ to position $j\neq i$ would require $n^{+}$ to change its sign. By a similar argument, we deduce that there are no edges between $U_{i}(-)$ and $U_{j}(-)$.

When $i<j$, for each vertex $u$ in $U_{i}(+)$ the number of neighbors of $u$ that are contained in $U_{j}(-)$ equals $\min(i,n-j+1)$. This follows because the only substring reversals that move the $n^+$ at index $i$ to the $n^-$ at index $j$ must reverse the  substring from index $i-x$ to $j+x$ for some $x \geq 0$. As $1 \leq i-x$ and $j+x \leq n$, there are just $\min(i,n-j+1)$ choices for $x$. As $i < j$, it follows that $\min(i,n-j+1) = \min(i,j,n-j+1,n-i+1)$ as required. Exactly the same argument applies when counting the number of neighbors of a vertex in $U_{i}(-)$ that are contained in $U_{j}(+)$. Next let $u$ be a vertex in $U_i(+)$, and consider the number of neighbors of $u$ in $U_i(-)$. These neighbors arise from $u$ by reversing a substring centered at position $i$, and there are $\min(i,n-i+1)$ such substrings. Since the graph $\sr{n}$ is regular of valency ${n+1\choose 2}$ we will get the entries on the diagonal such that each row sums to ${n+1\choose 2}$. 

From Chung and Tobin \cite{CT}, we get the eigenvalues of $D'+X$ as ${n+1\choose 2}-\lfloor \frac{k}{2}\rfloor n +2{\lfloor \frac{k}{2}\rfloor \choose 2}$ for $1\leq k\leq n$. If $w$ is an eigenvector of $D'+X$ with such eigenvalue $\mu$, then $\begin{bmatrix}w\\w\end{bmatrix}$ is an eigenvector of $\begin{bmatrix} D' & X\\
X& D'
\end{bmatrix}$ with eigenvalue $\mu$. By a similar argument, we can show that if $u$ is an eigenvector of $D'-X$ with eigenvalue $\theta$, then $\begin{bmatrix}u\\-u\end{bmatrix}$ is an eigenvector of $\begin{bmatrix} D' & X\\
X& D'
\end{bmatrix}$ with eigenvalue $\theta$. This determines $2n$ eigenvalues of the graph $\sr{n}$.

Another observation regarding the spectrum of $\sr{n}$ follows from Proposition \ref{equitablesrn1} and says that spectrum of $\sr{n}$ contains numbers of the form $\mu+n$, where $\mu$ is an eigenvalue of the reversal graph $R_n$. 
\end{proof}
Note that there is significant overlap between the two multisets of eigenvalues above. For example, ${n\choose 2}$ appears in both. We can observe this by noting that $e_i - e_{n+1-i}$, $1\leq i \leq \lfloor\frac{n}{2}\rfloor$ are eigenvectors of both $D'+X$ and $D'-X$, since they are orthogonal to $X$, and $D$ is persymmetric (i.e., symmetric about the principal back diagonal). The corresponding eigenvalues are the first $\lfloor\frac{n}{2}\rfloor$ diagonal entries of $D'$ and each of them (except the largest) is an eigenvalue of both $D'+X$ and $D'-X$.

\section{Open Problems}\label{sec:open}

In this paper, we studied the flip graph $\mathcal{M}(K_{2n})$ and the signed reversal graph $\sr{n}$. We improved some previous upper bounds for the chromatic number of $\mathcal{M}(K_{2n})$ and investigated the partition of its vertex set into signed reversal graphs and Cartesian products of signed reversal graphs. We also determined some properties of the signed reversal graphs such as part of their spectrum and some of their chromatic numbers for small order. There are several problems that are still open and we list them here.

\begin{enumerate}
\item Conjecture \ref{chiconj} is still open and it seems that its most difficult part is showing that 
$$\chi(\mathcal{M}(K_{2n}))\geq n+1.$$ 
This is open even for $n=5$. Proving that $\chi(\mathcal{M}(K_{2n}))$ is strictly increasing with $n$ would imply the above inequality.

\item A related open problem is investigating the independence number $\alpha(\mathcal{M}(K_{2n}))$ of the flip graph $\mathcal{M}(K_{2n})$. Our computations show that $\alpha(\mathcal{M}(K_8))=28$ and $\alpha(\mathcal{M}(K_{10}))\geq 208$. Using Hoffman's ratio bound, one can get the general, but weak bound of $\alpha(\mathcal{M}(K_{2n}))\leq \frac{(2n-1)!!}{3}$ for any $n\geq 2$.

\item The second largest eigenvalue of a regular graph is related to its connectivity and expansion properties and has been determined for various nice regular graphs including several Cayley graphs of the symmetric group (see \cite{CT,HHC} for example). Proposition \ref{eigsr} implies that ${n\choose 2}$ is an eigenvalue of $\sr{n}$ and our computations for $n\leq 5$ suggest the following conjecture.
\begin{conj}
For $n\geq 2$, ${n\choose 2}$ is the second largest eigenvalue of the signed reversal graph $\sr{n}$.
\end{conj}

\item Our computational results determining (or bounding) the chromatic number of $\sr{n}$ for small $n$ are shown in Table~\ref{chromaticsrn}. What is the behavior of this chromatic number as $n$ increases?


\end{enumerate}

\begin{appendices}

\section{$\chi(\sr5)\leq 4$}

We list below the $4$-coloring of $\sr5$ obtained from our code. For a permutation $\pi=(\pi_1\dots\pi_k)\in S_k$ we denote by $\pi_1\dots \pi_k^{0}$ (or $\pi_1\dots \pi_k^{1}$) the subset of vertices of $\sr5$ whose underlying permutation is $(\pi_1\dots \pi_k)$ that have an even (or odd) number of $+$s. Equivalently, these are the color classes of the $k$-dimensional cube induced in $\sr{k}$ by $V_{\pi}$.\\

Color 0: [$52143^{0}$, $12345^{0}$, $43251^{0}$, $42315^{0}$, $34521^{0}$, $13452^{0}$, $13254^{0}$, $35412^{0}$, $21435^{0}$, $25413^{0}$, $24531^{0}$, $42153^{0}$, $54123^{0}$, $45213^{0}$, $51432^{0}$, $35214^{0}$, $31524^{0}$, $34215^{0}$, $31245^{0}$, $32541^{0}$, $15324^{0}$, $54231^{0}$, $51234^{0}$, $41235^{0}$, $23451^{0}$, $14325^{0}$, $53421^{0}$, $12453^{0}$, $52413^{1}$, $41523^{1}$, $51423^{1}$, $32154^{1}$, $54312^{1}$, $51342^{1}$, $24513^{1}$, $13425^{1}$, $35142^{1}$, $23145^{1}$, $52314^{1}$, $43125^{1}$, $12534^{1}$, $14532^{1}$, $24351^{1}$, $21543^{1}$, $25341^{1}$, $45321^{1}$, $34152^{1}$, $42135^{1}$, $13524^{1}$, $45132^{1}$, $25134^{1}$, $14235^{1}$, $23514^{1}$, $43512^{1}$, $41352^{1}$, $21354^{1}$, $34251^{1}$, $35241^{1}$]\\

Color 1:   [$54213^{1}$, $24153^{0}$,  $14523^{1}$,  $45123^{0}$,  $12354^{0}$,  $51243^{1}$,  $54321^{1}$,  $53412^{1}$,  $41253^{1}$,  $51324^{1}$,  
$42513^{1}$,  $32415^{0}$,  $15342^{0}$,  $13245^{1}$,  $32145^{1}$,  $31542^{0}$,  $31425^{0}$,  $54132^{0}$,  $41325^{0}$,  $53214^{1}$,  
$42351^{1}$,  $23541^{0}$,  $15432^{1}$,  $34125^{1}$,  $25143^{0}$,  $43521^{0}$,  $15234^{1}$,  $45231^{1}$,  $14352^{0}$,  $31254^{1}$,  
$52134^{0}$,  $23415^{1}$,  $35124^{1}$,  $52341^{1}$,  $12543^{1}$,  $43215^{1}$,  $25431^{1}$,  $21345^{0}$,  $31452^{1}$,  $21453^{1}$,  
$32514^{0}$,  $24315^{0}$,  $52431^{0}$,  $34512^{0}$,  $53124^{0}$,  $13542^{1}$,  $42531^{0}$,  $23154^{0}$,  $43152^{0}$,  $12435^{1}$,  
$53241^{0}$,  $21534^{0}$,  $41532^{0}$,  $25314^{0}$,  $45312^{0}$,  $15243^{0}$,  $14253^{0}$,  $35421^{1}$]   \\
   
Color 2: [$15423^{0}$,  $52314^{0}$,  $42513^{0}$,  $51243^{0}$,  $14253^{1}$,  
$54231^{1}$,  $54123^{1}$,  $41523^{0}$,  $45213^{1}$,  $51234^{1}$,  
$15243^{1}$,  $54213^{0}$,  $52143^{1}$,  $25413^{1}$,  $12534^{0}$,  
$35412^{1}$,  $51324^{0}$,  $21543^{0}$,  $31425^{1}$,  $32451^{0}$,  
$15342^{1}$,  $14235^{0}$,  $31542^{1}$,  $24153^{1}$,  $13425^{0}$,  
$23145^{0}$,  $43125^{0}$,  $45321^{0}$,  $41235^{1}$,  $12345^{1}$,  
$14325^{1}$,  $43251^{1}$,  $13254^{1}$,  $34521^{1}$,  $42351^{0}$,  
$24315^{1}$,  $25341^{0}$,  $35214^{1}$,  $35421^{0}$,  $32154^{0}$,  
$35124^{0}$,  $14532^{0}$,  $21435^{1}$,  $23451^{1}$,  $13452^{1}$,  
$32541^{1}$,  $51432^{1}$,  $34152^{0}$,  $53421^{1}$,  $54312^{0}$,  
$32415^{1}$,  $24135^{0}$,  $13542^{0}$,  $53124^{1}$,  $21534^{1}$,  
$41532^{1}$,  $25314^{1}$,  $45312^{1}$,  $42531^{1}$,  $45132^{0}$,  
$25134^{0}$,  $53142^{0}$]\\

Color 3: [$24513^{0}$,  $32451^{1}$,  $45123^{1}$,  $51423^{0}$,  $15423^{1}$,
$15432^{0}$,  $32514^{1}$,  $12453^{1}$,  $53214^{0}$,  $52341^{0}$,  
$12543^{0}$,  $25143^{1}$,  $21453^{0}$,  $41253^{0}$,  $51342^{0}$,  
$34251^{0}$,  $43215^{0}$,  $52413^{0}$,  $15324^{1}$,  $12435^{0}$,  
$31245^{1}$,  $31452^{0}$,  $34512^{1}$,  $54132^{1}$,  $54321^{0}$,  
$35241^{0}$,  $35142^{0}$,  $42153^{1}$,  $31524^{1}$,  $34125^{0}$,  
$13245^{0}$,  $41325^{1}$,  $53241^{1}$,  $15234^{0}$,  $23415^{0}$,  
$34215^{1}$,  $14352^{1}$,  $52134^{1}$,  $42315^{1}$,  $53412^{0}$,  
$31254^{0}$,  $52431^{1}$,  $21345^{1}$,  $23541^{1}$,  $14523^{0}$,  
$45231^{0}$,  $12354^{1}$,  $32145^{0}$,  $24351^{0}$,  $43521^{1}$,  
$25431^{0}$,  $53142^{1}$,  $24531^{1}$,  $24135^{1}$,  $13524^{0}$,  
$42135^{0}$,  $23514^{0}$,  $43512^{0}$,  $21354^{0}$,  $41352^{0}$,  
$23154^{1}$,  $43152^{1}$]

\end{appendices}

\section*{Acknowledgments} We are very grateful to the anonymous referees, Ferdinand Ihringer, Nathan Lindzey, Chia-an Liu, Jack Koolen and Josh Tobin for their comments and suggestions.

\end{document}